\documentclass[12pt]{article}
\usepackage[a4paper,margin=1in]{geometry}
\usepackage[T1]{fontenc}
\usepackage[utf8]{inputenc}
\usepackage{amsmath,amssymb,amsthm,mathtools}
\usepackage{array}
\usepackage{booktabs}
\usepackage{enumitem}

\usepackage{shuffle}
\usepackage{hyperref}

\newcommand{\Sh}{\operatorname{\bf Sh}}
\newcommand{\Des}{\operatorname{Des}}
\newcommand{\IDes}{\operatorname{IDes}}
\newcommand{\des}{\operatorname{des}}
\newcommand{\maj}{\operatorname{maj}}

\newcommand{\sgn}{\operatorname{sgn}}
\newcommand{\nfix}{\operatorname{nfix}}
\newcommand{\nmark}{\operatorname{nmark}}

\newtheorem{con}{Conjecture}[section]

\newtheorem{lem}[con]{Lemma}
\newtheorem{prop}[con]{Proposition}

\newtheorem{thm}[con]{Theorem}

\newtheorem{pf}{proof}[section]
\theoremstyle{remark}

\makeatletter \@addtoreset{equation}{section} \makeatother
\makeindex 
\def\qed{\hfill \rule{4pt}{7pt}}

\begin{document}
 
\begin{center}
{\Large \bf Revisiting $q$-Derangement Numbers via\\[5pt] Decorated Permutations}

\vskip 4mm
Kathy Q. Ji
\vskip 2mm

Center for Applied Mathematics\\[2pt]
Tianjin University\\[2pt]
Tianjin 300072, P.R. China

\vskip 2mm
kathyji@tju.edu.cn

\end{center}

\vskip 6mm \noindent {\bf Abstract.}  This paper aims to provide a direct combinatorial proof of the Gessel-Wachs formula for
$q$-derangement numbers in the setting of decorated permutations, without using
the $q$-binomial inversion formula. Decorated permutations, introduced by
Postnikov in his study of the totally nonnegative Grassmannian, provide a natural
framework for Chen's signed fixed-point model. Our proof is based on a
major-index generating function for decorated permutations with a fixed
number of  signed fixed points, together with a sign-reversing and
descent-set-preserving involution, thereby answering a question raised by Chen. This involution was discovered through human--AI collaboration. Moreover, our framework yields an immediate proof of  a result of D\'esarm\'enien and Wachs  concerning the equidistribution between descent classes of derangements and descent classes of desarrangements. We also supply combinatorial proofs  of two recurrence relations for the $q$-derangement numbers in the setting of   desarrangements with the aid of the insertion lemma for ordinary permutations.



\section{Introduction}

We will follow the terminology and notation on permutations and partitions   in Andrews \cite{Andrews-1976}  and Stanley \cite{Stanley-2012}.   Let $\mathfrak{S}_n$ denote the set of all permutations on $[n]:=\{1,2,\ldots, n\}$. An integer $i$ with $1 \leq i \leq n$ is said to be a \textit{fixed point} of $\pi=\pi_1\cdots \pi_n \in \mathfrak{S}_n$ if $\pi_i = i$, and a \emph{nonfixed point} otherwise.  \textit{Derangements} are permutations with no fixed points. Let $\mathcal{D}_n$ be the set of all derangements in $\mathfrak{S}_n$ and let $d_n$ denote the number of derangements in $\mathfrak{S}_n$. It is well-known that 
\begin{equation}\label{eq:derang}
d_n = \sum_{k=0}^{n} (-1)^k \frac{n!}{k!}.
\end{equation}

In \emph{An Undergraduate Course in Combinatorics}, William Y. C. Chen gave a
simple combinatorial proof of \eqref{eq:derang} by considering permutations in
which fixed points may be assigned a minus sign. It is worth noting that Chen's signed fixed-point model happens to coincide with
 decorated permutations, namely permutations in which each fixed
point is assigned one of two colors.

  Decorated permutations were introduced by Postnikov \cite{Postnikov2006}
in his study of the totally nonnegative Grassmannian. Corteel~\cite{Corteel2007} and Williams~\cite{Williams2005}  further studied
decorated permutations in connection with permutation statistics such as weak
excedances and alignments. For the definitions of weak excedances and alignments, see \cite{Corteel2007,Williams2005}. More generally,
Blitvi{\'c} and Steingr{\'i}msson \cite{BlitvicSteingrimsson2021} introduced $k$-arrangements, namely
permutations whose fixed points are colored with one of $k$ colors;   Thus decorated permutations are precisely
$2$-arrangements; see also Fu, Han, and Lin~\cite{Fu-Han-Lin-1979}.

Let $E_n$ be the set of decorated permutations on $[n]$. A signed fixed point is denoted by $\overline{i}$. 
For   $0\leq k\leq n$, let $E_{n,k}$ denote the subset
of $E_n$ with exactly $k$ signed fixed points. Clearly, 
\begin{equation} \label{eqaa}
|E_{n,k}| = \binom{n}{k} (n-k)! = \frac{n!}{k!}.
\end{equation}
Consider the subset of $E_n$ consisting of permutations with at least one fixed point, signed or unsigned. Let $i$ be the minimum fixed point (in absolute value). Changing the sign of $i$ gives a sign-reversing involution. After the cancellation, only ordinary derangements remain, which yields \eqref{eq:derang}.

William Y. C. Chen posed the natural question of whether this signed fixed-point
framework can also be applied to the $q$-derangement numbers; see
\cite{Chen2026} for details.

To state the formula for the $q$-derangement numbers, we first recall the
definitions of the descent set and the major index on ordinary permutations. For $\pi=\pi_1\cdots \pi_n \in \mathfrak{S}_n$, we say that $1 \leq i \leq n - 1$ is a descent   if $\pi_i > \pi_{i+1}$ and $1 \leq i \leq n - 1$ is an ascent   if $\pi_i < \pi_{i+1}$. The set of descents of $\pi$ is called the descent set of $\pi$, denoted $\operatorname{Des}(\pi)$ and the number of its descents is called the descent number, denoted $\operatorname{des}(\pi)$. The major index of $\pi$, denoted $\operatorname{maj}(\pi)$, is defined to be the sum of its descents. To wit,
\[
\operatorname{maj}(\pi) := \sum_{k\in\operatorname{Des}(\pi)} k.
\]
The following formula due to MacMahon  is well-known:
 
\begin{equation}\label{eq:MacM}
    \sum_{\pi \in \mathfrak{S}_n}q^{{\rm maj}(\pi)}=[n]_q!=[1]_q[2]_q\cdots [n]_q.
\end{equation} 
Here and in the sequel, $[0]_q!=1$ and for  a positive integer $n$, we define
\[[n]_q:=\frac{1-q^n}{1-q}=1+q+\cdots +q^{n-1}\]

The $q$-derangement numbers are defined by $d_0(q) = 1$ and for $n \geq 1$, 
\[
d_n(q) = \sum_{\sigma \in \mathcal{D}_n} q^{\operatorname{maj}(\sigma)}.
\]

The following elegant formula was first derived by Gessel and published in  \cite{Gessel-Reutenauer-1993}  as a  consequence of the quasi-symmetric generating function encoding the descents and the cycle structure of permutations. Note that  setting $q=1$  in \eqref{eq:q-derange} recovers the classical derangement formula \eqref{eq:derang}.
 
\begin{thm}[Gessel-Wachs]\label{thm:q-derange} For $n\geq 1$, 
 \begin{equation} \label{eq:q-derange}
d_n(q) = \sum_{k=0}^n (-1)^k q^{k \choose 2} { [n]_q! \over [k]_q!}, 
 \end{equation} 
 \end{thm}

A combinatorial proof of \eqref{eq:q-derange} has been obtained by Wachs  \cite{Wachs1989}. Let us first review the combinatorial settings of Wachs.  Let $A=\{0<a_1<a_2<\cdots <a_n\}$ and let $\mathfrak{S}_A$ denote the set of permutations of the set $A$. For
 $\pi=\pi_1\cdots \pi_n \in \mathfrak{S}_A$, the  reduction of $\pi$ is the permutation in $\mathfrak{S}_n$ by replacing each letter $a_j$  by $j$.  For example, $\pi = 9\, 3\, 8\, 10\, 12\, 2\, 7$  is the permutation of the set $A=\{2,\,3,\,7,\,8,\,9,\,10,\,12\}$, its reduction is $5\, 2\, 4\, 6 \,7\, 1\, 3$, which is a permutation in  $\mathfrak{S}_7$.

 Let $\pi=\pi_1\cdots\pi_n \in \mathfrak{S}_n$. The derangement part of $\pi$, denoted $dp(\pi)$, is the reduction of the subword of non-fixed points of $\pi$. Recall that $\pi_i$ is called the non-fixed points of $\pi$ if $\pi_i \neq i$. For example, let's take the following permutation in $\mathfrak{S}_9$:
\begin{equation}\label{exa-s}
\pi = 1\, 5\, 3\, 7\, 6\, 2\, 9\, 8\, 4,
\end{equation}
there are three fixed points, which are $1,3,8$ and six non-fixed points:  $5,\,7,\,6,\,2,\,9,\,4$. The reduction of non-fixed points of $\pi$ is $3\, 5\, 4\, 1\,6\, 2$, so the derangement part of $\pi$ is
\[dp(\pi)= 3\, 5\, 4\, 1\,6\, 2.\]

 Wachs \cite{Wachs1989}   established the following relation: 
\begin{prop}[Wachs]\label{DSn}
Let $0\leq k\le n$ and $\sigma \in \mathcal{D}_k$. We have
\begin{equation}\label{DSn-eq}
    \sum_{dp(\pi)=\sigma \atop \pi \in \mathfrak{S}_n} q^{{\rm maj}(\pi)}=q^{{\rm maj}(\sigma)}{n \brack k}_q,
\end{equation}
where
\[{n \brack k}_q=\frac{[n]_q!}{[k]_q![n-k]_q!}\]
is the $q$-binomial coefficients.
\end{prop}
Summing over all derangements $\pi \in \mathcal{D}_k$ and $0 \leq k \leq n$, and applying \eqref{eq:MacM}, we can deduce from \eqref{DSn-eq} that
\[
[n]! = \sum_{k=0}^{n} {n\brack k}_q d_k(q).
\]
Thus \eqref{eq:q-derange} follows from the $q$-binomial inversion \cite[Corollary 3.38]{Aigner-1979}.

In order to justify the relation \eqref{DSn-eq}, Wachs \cite{Wachs1989} found a bijection on $\mathfrak{S}_n$ by rearranging a permutation $\pi$ according to \textit{excedant} ($\pi_i > i$), fixed point, and \textit{subcedant} ($\pi_i < i$). She showed that this bijection preserves the major index. Then the following result of Garsia-Gessel \cite[Theorem 3.1]{Garsia-Gessel-1979} on shuffles of permutations is applied to establish Theorem  \ref{thm:q-derange}.

Let $\pi \in \mathfrak{S}_j$ and $\delta \in \mathfrak{S}_k$ be two disjoint permutations, that is, permutations with no letters in common. We say that $\alpha \in \mathfrak{S}_{j+k}$ is a shuffle of $\pi $ and $\delta$ if both $\pi$ and $\delta$ are subsequences of $\alpha$. The set of shuffles of $\pi $ and $\delta$ is denoted $\pi \shuffle \delta$.  For example, let $\pi=263$ and $\delta=14$, we have 
\[
263\shuffle 14 = \{26314, 26134, 26143, 21463, 21634, 21643, 12463, 14263, 12634, 12643\}. \]
 
\begin{thm}[Garsia-Gessel] Let $\pi \in \mathfrak{S}_m$ and $\delta \in \mathfrak{S}_n$ be two disjoint permutations. We have 
\begin{equation}\label{Gasia-Gessel}
\sum_{ \alpha \in  \pi\shuffle \delta} q^{\operatorname{maj}(\alpha)}
= q^{\operatorname{maj}(\pi)+\operatorname{maj}(\delta)} {n+m \brack n}_q.  
\end{equation}
\end{thm}

Inspired by MacMahon's original proof of \eqref{eq:MacM}, Chen and Xu \cite{Chen-Xu-2008}  present an alternative approach to Wachs' formula \eqref{DSn-eq} based on the following reformulation   in terms of labeled partitions: 
\begin{equation}
\frac{1}{(q)_n} \sum_{\substack{\pi \in \mathfrak{S}_n \\ dp(\pi)=\sigma}} q^{\operatorname{maj}(\pi)}
= \frac{1}{(q)_k (q)_{n-k}} q^{\operatorname{maj}(\sigma)}.
\label{eq:1.6}
\end{equation}

 William Y. C. Chen raised the problem of finding a direct proof of
Theorem~\ref{thm:q-derange} in the setting of decorated permutations, without
using the $q$-binomial inversion formula. As a first step, Chen extended the descent set and the major index from ordinary
permutations to decorated permutations in $E_n$. This is done by regarding a
decorated permutation as a word in the alphabet
\begin{equation}\label{order}
\bar n<\cdots<\bar2<\bar1<1<2<\cdots<n.
\end{equation}
For example,
\[
\Des(2\,1\,3\,\overline{4}\,6\,5\,\overline{7})=\{1,3,5,6\},
\quad 
\maj(2\,1\,3\,\overline{4}\,6\,5\,\overline{7})=15.
\]

Chen~\cite{Chen2026} observed the following formula for the major-index
generating function over decorated permutations with a fixed number of
 signed fixed points. When $q=1$, this formula reduces to \eqref{eqaa}.

\begin{thm}\label{main1}
For $0\leq k\leq n$, let $E_{n,k}$ denote the set of 
decorated permutations in $E_n$ with exactly $k$  signed fixed points. We have 
\begin{equation}\label{eq:main1}
\sum_{\pi\in E_{n,k}} q^{\maj(\pi)}
=
q^{\binom{k}{2}}\frac{[n]_q!}{[k]_q!}.
\end{equation}
\end{thm}

As mentioned by Chen~\cite{Chen2026},  Catherine Yan established a proof of \eqref{eq:main1} by adapting the insertion lemma for classical permutations to decorated permutations, while Peter Guo gave a proof using Stanley's theory of
$P$-partitions. For more information on the insertion lemma for ordinary permutations, we refer to 
Gupta \cite{Gupta-1978},  Haglund, Loehr and Remmel \cite{Haglund-Loehr-Remmel-2005} and Han \cite{Han-1992}. 

In Section~2, we provide a simple direct proof based on  MacMahon’s major index formula \eqref{eq:MacM}  
and the Garsia--Gessel shuffle formula \eqref{Gasia-Gessel}.

Chen further asked for a sign-reversing and descent-set-preserving involution on
$G_n=E_n\setminus \mathcal{D}_n,$ where $\mathcal{D}_n$ is regarded as the set of ordinary unsigned derangements
embedded in $E_n$. More precisely, 
\begin{thm}\label{main2}
Let $G_n$ denote the set of decorated permutations on $[n]$  with at least one
fixed point, either signed or unsigned. For $\pi\in G_n$, let $\nfix(\pi)$ denote the number of
 signed fixed points of $\pi$. We have 
\begin{equation}
\sum_{\pi\in G_n}(-1)^{\nfix(\pi)}q^{\maj(\pi)}=0.
\end{equation}
\end{thm}
Combining Theorems~\ref{main1} and~\ref{main2} immediately yields
Theorem~\ref{thm:q-derange}. In Section~3, we construct the desired
sign-reversing and descent-set-preserving involution by extending Wachs' setting to decorated permutations.

Last but not least, we would like to mention that the framework developed in this paper immediately yields a result of D\'esarm\'enien and Wachs \cite{DesarmenienWachs1988, DesarmenienWachs1993} concerning the equidistribution between descent classes of derangements and descent classes of desarrangements. Recall that the notion of a \textit{desarrangement} was  introduced to give a combinatorial explanation of the  recurrence relation for the derangement numbers $d_n$. A desarrangement is  
a permutation whose first ascent is even. Let $A^e_n$ denote the set of all desarrangements in $\mathfrak{S}_n$.
For example, $A^e_1 = \emptyset$, $A^e_2 = \{21\}$, $A^e_3 = \{213, 312\}$ and 
\[A^e_4 = \{3241, 3142, 2143, 2134, 4231, 4132, 4123, 3124, 4321\}.\]

We write $\delta_n$ for the decreasing permutation $n, n-1, \dots, 1$. Observe that $\delta_n \in A_n^e$ if and only if $n$ is even.

D\'esarm\'enien and Wachs \cite{DesarmenienWachs1988, DesarmenienWachs1993} introduced the $q$\textit{-derangement} numbers defined via
\begin{equation}\label{defi:desar} 
e_n(q)=\sum_{\sigma\in A^e_n} q^{\operatorname{maj}(\sigma^{-1})},
\end{equation}
and established the identity
\begin{equation}\label{eq:2-3}
d_n(q)=e_n(q).
\end{equation}

In fact, they proved the following general statement.

\begin{prop}[D\'esarm\'enien--Wachs]\label{prop-deradesarra}  
For every subset $J \subseteq \{1, 2, \dots, n-1\}$,
\begin{equation}\label{prop-deradesarra-eq}
\bigl|\{\sigma \in \mathcal{D}_n \mid \Des(\sigma) = J\}\bigr|
= \bigl|\{\sigma \in A^e_n \mid \IDes(\sigma) = J\}\bigr|,
\end{equation}
where $\IDes(\sigma) \coloneqq \Des(\sigma^{-1})$.
\end{prop}

In \cite{DesarmenienWachs1993}, the authors asked for a direct bijection of \eqref{prop-deradesarra-eq}  between these sets.  In Section 4, we show that our framework yields an immediate proof of this proposition. Additionally, we supply combinatorial proofs for the following recurrence relations satisfied by the $q$-derangement numbers $d_n(q)$, leveraging the identity \eqref{eq:2-3} together with the insertion lemma for ordinary permutations.

\begin{prop}\label{prop:recurr} For all integers $n\geq 1$, 
\begin{align}
d_n(q) &= [n]_q d_{n-1}(q) + (-1)^n q^{\binom{n}{2}}, \label{rec-eqa} \\[5pt] 
d_n(q) &= [n-1]_q\big(d_{n-1}(q) + q^{n-1} d_{n-2}(q)\big).  \label{rec-eqb}
\end{align}
\end{prop}

These recurrences are $q$-analogues of the well-known classical recursions for ordinary derangement numbers $d_n$. They may be readily deduced from the Gessel-Wachs formula \eqref{eq:q-derange} for $d_n(q)$. A combinatorial proof of \eqref{rec-eqa} previously appeared in \cite{DesarmenienWachs1993},   relying on the equality \eqref{eq:2-3}, where the argument employs Foata's second fundamental transformation. By contrast, our proof of \eqref{rec-eqa} relies  on the insertion lemma.

\section{Proof of Theorem \ref{main1} }

Given $0\leq k\leq n$, let $\mathfrak{S}(k+1,\ldots, n)$ denote the set of permutations on $\{k+1,\ldots, n\}$, and let $\delta_k=(k,k-1,\ldots 1)$ with the convention that $\delta_0$ is the empty word.   There is a canonical descent-preserving bijection
\[
  \phi\colon E_{n,k}\longrightarrow \cup_{\beta \in \mathfrak{S}(k+1,\ldots, n)} \delta_k \shuffle \beta
\]
such that $\Des(\phi(\pi))=\Des(\pi)$ for every $\pi\in E_{n,k}$. 

Using this natural bijection $\phi$, we give a short direct proof of Theorem \ref{main1}, relying on  MacMahon’s major index formula \eqref{eq:MacM} and the Garsia-Gessel shuffle identity \eqref{Gasia-Gessel}.

\begin{proof}[Proof of Theorem \ref{main1}] By the descent-set-preserving bijection $\phi$, we derive that 
\begin{align*}
\sum_{\pi\in E_{n,k}}q^{\maj(\pi)}
&= \sum_{\beta\in \mathfrak{S}{\{k+1,\ldots,n\}}}
\sum_{\phi(\pi) \in \delta_k \shuffle \beta }q^{\maj(\phi(\pi))}\\[5pt]
&\overset{\eqref{Gasia-Gessel}}{=}\sum_{\beta\in \mathfrak{S}{\{k+1,\ldots,n\}}}q^{\maj(\delta_k)+\maj(\beta)}{n\brack k}_q\\[5pt]
&{=}q^{{k\choose 2}}{n\brack k}_q\sum_{\beta\in \mathfrak{S}{\{k+1,\ldots,n\}}}q^{\maj(\beta)}\\[5pt]
&\overset{\eqref{eq:MacM}}{=}q^{{k\choose 2}}{n\brack k}_q[n-k]_q!=q^{\binom{k}{2}}\frac{[n]_q!}{[k]_q!},
\end{align*}
 as desired. 
\end{proof}

\section{An involution}

In this section, we aim to prove Theorem \ref{main2} by constructing the sign-reversing and descent-set-preserving involution on decorated permutations in the Wachs' setting.  In fact, this involution establishes the following stronger refined identity: 

\begin{thm}\label{main2a} For $0\leq k\leq n-1$, let  $\sigma \in {\mathcal{D}}_k$ and let $E_n(\sigma)$ denote the set of  permutations $\pi$ in $E_n$ such that $dp(\pi)=\sigma$, we have 
\begin{equation}
\sum_{\pi\in E_n(\sigma)}(-1)^{\nfix(\pi)}z^{\des(\pi)}q^{\maj(\pi)}=0.
\end{equation}
\end{thm}
Summing the above identity over all $\sigma\in \mathcal{D}_k$ for $0 \leq k \leq n-1$ immediately recovers   Theorem \ref{main2}.

We first extend Wachs’ combinatorial setting to  decorated permutations.  For $\pi = \pi_1  \ldots \pi_n \in E_n$, we say that a letter $\pi_i$ of $\pi$ is an excedant (resp. subcedant) of $\pi$ if $\pi_i > i$ (resp. $\pi_i < i$). Let $s(\pi)$ and $e(\pi)$ denote the numbers of subcedants and excedants of $\pi$ respectively.   We now fix $n$ and let $k \leq n$.  {\it The map ${\psi}_{n}$} is defined as follows: Let $\pi \in E_k$, then  ${\psi}_{n}(\pi)=\tilde{\pi}$ is obtained from $\pi$ by replacing its $i$th smallest subcedant $\pi_j$ by $i$, $i = 1, 2, \ldots, s(\pi)$, its $i$th smallest  (in absolute value) fixed point $\pi_j$ by $(\sgn(\pi_j))(s(\pi) + i)$, $i = 1, 2, \ldots, k - s(\pi) - e(\pi)$, and its $i$th largest excedant by $n - i + 1$, $i = 1, 2, \ldots, e(\pi)$.

  For example, let $\pi = 5\, 8\, \bar{3}\, 1\, 6\, 4\, 7\, 2 \in E_8$. We see that $s(\pi)=3$, $e(\pi)=3$ and $\sigma=dp(\pi)=4\, 6\,  1\, 5\,3\,2$. Then 
  
 \[{\psi}_{8}({\sigma})=6\, 8\,  1\, 7\,3\,2,\quad  {\psi}_{8}({\pi})=6\, 8\, \bar{4}\, 1\, 7\, 3\, 5\, 2.\]

For $\sigma \in {\mathcal{D}}_k$, let $E_n(\sigma)$ denote the set of  permutations $\pi$ in $E_n$ such that $dp(\pi)=\sigma$, and 
let $\Gamma(\sigma)$ be the set of all signed sequences $(a_1,\ldots, a_{n-k})$, where $a_i=s(\sigma)+i$ or  $\overline{s(\sigma)+i}$.     We have the following consequence: 

\begin{lem}\label{lem:inv}  For $0\leq k\leq n$, let  $\sigma \in {\mathcal{D}}_k$ and set $\tilde{\sigma}={\psi}_{n}(\sigma)$.  Let $\Sh(\tilde{\sigma},\Gamma(\sigma))$ denote the shuffles of $\tilde{\sigma}$ and $\gamma \in  \Gamma(\sigma)$. Then the map ${\psi}_{n}$ provides a bijection between $E_n(\sigma)$ and $\Sh(\tilde{\sigma},\Gamma(\sigma))$. Moreover, for $\pi \in E_n(\sigma)$, we have  $\psi_n(\pi)\in\Sh(\widetilde{\sigma},\gamma)$ for some
$\gamma\in\Gamma(\sigma)$ such that
 \begin{equation}\label{lem:rel}
 \Des(\pi)=\Des({\psi}_{n}(\pi)), \quad  \nfix(\pi)=\nmark(\gamma),
 \end{equation}
 where $\nmark(\gamma)$ denotes the number of marked terms in $\gamma$. 
\end{lem}

\begin{proof} 
By the construction of $\psi_n$, it is straightforward to verify that $\psi_n$
gives a bijection between $E_n(\sigma)$ and
$\Sh(\widetilde{\sigma},\Gamma(\sigma))$. It remains to prove the two
relations in \eqref{lem:rel}. Wachs proved that the corresponding map preserves descent sets in the unsigned
case. We claim that the signed extension considered here also preserves descent
sets. Indeed, when two adjacent letters are both positive, the assertion is
exactly Wachs' result. If one of the adjacent letters is a minus-signed fixed
point, then this letter is smaller than every positive letter with respect to
the order \eqref{order}. Thus the descent status of such an adjacent pair is unchanged under $\psi_n$.
Finally, if both adjacent letters are minus-signed fixed points, then their
relative order is reversed by the above signed order, and the same reversal is
reflected in the order of the fixed-point word $\gamma \in \Gamma(\sigma)$. Hence
$\psi_n$ preserves descent sets, that is,
$\Des(\pi)=\Des(\psi_n(\pi)).$ Moreover,   $\nfix(\pi)=\nmark(\gamma).$ This completes the proof.
\end{proof}

Armed with Lemma \ref{lem:inv}, we are now  in a position to prove Theorem \ref{main2a}. The main idea of the proof is to flipping the sign and sliding the position of the first term of $\gamma$ in $\Sh(\tilde{\sigma},\Gamma(\sigma))$.

\begin{proof}[Proof of Theorem \ref{main2a}] Given $\sigma \in \mathcal{D}_k$ for $0\leq k\leq n-1$, it suffices to construct an involution $\Psi$ on $E_n(\sigma)$ 
such that, for every $\pi\in E_n(\sigma)$,
\begin{equation}\label{pf-rel} 
  (-1)^{\nfix(\Psi(\pi))}=-(-1)^{\nfix(\pi)},\qquad
  \Des(\Psi(\pi))=\Des(\pi).
\end{equation}

Let $\pi \in E_n(\sigma)$ and set $\tilde{\pi}={\psi}_{n}(\pi)$. By Lemma \ref{lem:inv}, we see that $\tilde{\pi}\in \Sh(\tilde{\sigma},\Gamma(\sigma))$, where $\tilde{\sigma}={\psi}_{n}(\sigma)$.  Thus we may write $\tilde{\pi}=\tilde{\sigma} \shuffle \gamma $, where $\gamma=\gamma_1\ldots \gamma_{n-k} \in \Gamma(\sigma)$. Note that  $\gamma_1=\overline{s(\sigma)+1}$ or $s(\sigma)+1$.
Suppose $\tilde{\pi}=\tilde{\pi}_1\ldots \tilde{\pi}_n$ and let $j$ be the index such that  $\tilde{\pi}_j=\gamma_1$. We further adopt the convention $\tilde{\pi}_0=\tilde{\pi}_{n+1}=+\infty$.

We now define a word $\widetilde{\tau}$ by changing the sign of the distinguished
letter $\gamma_1$ and, when necessary, sliding it across a maximal string of
letters in the interval $(\overline{s(\sigma)+1},s(\sigma)+1)$ so that the descent set is preserved. Note that all intervals below are taken with
respect to the   order given as   \eqref{order}. Finally,   define 
\[
\Psi(\pi)=\tau=\psi_n^{-1}(\widetilde{\tau}).\]

We distinguish three cases: 

\begin{enumerate}
\item Suppose that $\gamma_1=\overline{s(\sigma)+1}$, that is $\tilde{\pi}_j=\overline{s(\sigma)+1}$. We consider the following two subcases:
\begin{enumerate}

\item If neither $\tilde{\pi}_{j-1}$ or   $\tilde{\pi}_{j+1}$ belongs to    $(\overline{s(\sigma)+1}, {s(\sigma)+1})$, then we define
$\widetilde{\tau}$ by replacing $\overline{s(\sigma)+1}$ in $\widetilde{\pi}$ by ${s(\sigma)+1}$:
\[
\widetilde{\tau}
=
\widetilde{\pi}_1\cdots
\widetilde{\pi}_{j-1}\,
{(s(\sigma)+1)}\,
\widetilde{\pi}_{j+1}\cdots
\widetilde{\pi}_n.
\]. 

\item Suppose that $\tilde{\pi}_{j-1} \in (\overline{s(\sigma)+1}, {s(\sigma)+1})$. There are two cases: 

\begin{enumerate}

\item If  $\tilde{\pi}_{j-1}< \tilde{\pi}_{j+1}$ or $\tilde{\pi}_{j+1}=\overline{s(\sigma)+2}$,  then choose maximal $r$  such that $1\leq r\leq j-1$ and
$\widetilde{\pi}_{j-i}\in
(\widetilde{\pi}_{j-i+1},s(\sigma)+1)$ {for } $1\leq i\leq r.$ Equivalently, $\widetilde{\pi}_{j-r-1}\notin(\widetilde{\pi}_{j-r},s(\sigma)+1)$.  Define
\[
\widetilde{\tau}
=
\widetilde{\pi}_1\cdots
\widetilde{\pi}_{j-r-1}\,
(s(\sigma)+1)\,
\widetilde{\pi}_{j-r}\cdots
\widetilde{\pi}_{j-1}
\widetilde{\pi}_{j+1}\cdots
\widetilde{\pi}_n.
\]

\item If  $\tilde{\pi}_{j-1}> \tilde{\pi}_{j+1}$ and $\tilde{\pi}_{j+1}\neq \overline{s(\sigma)+2}$,   then choose maximal  $r$ such that $1\leq r\leq n-j$ and
$\widetilde{\pi}_{j+i}\in
(\widetilde{\pi}_{j+i-1},s(\sigma)+1)$ for $ 1\leq i\leq r.$ Equivalently, $\widetilde{\pi}_{j+r+1}\notin
(\widetilde{\pi}_{j+r},s(\sigma)+1)$. Define
\[
\widetilde{\tau}
=
\widetilde{\pi}_1\cdots
\widetilde{\pi}_{j-1}
\widetilde{\pi}_{j+1}\cdots
\widetilde{\pi}_{j+r}\,
(s(\sigma)+1)\,
\widetilde{\pi}_{j+r+1}\cdots
\widetilde{\pi}_n.
\]

\end{enumerate}
 
\item Suppose that  $\tilde{\pi}_{j+1} \in (\overline{s(\sigma)+1}, {s(\sigma)+1})$. 

\begin{enumerate}

\item If  $\tilde{\pi}_{j-1}< \tilde{\pi}_{j+1}$, and so $\tilde{\pi}_{j-1} \in   (\overline{s(\sigma)+1}, {s(\sigma)+1})$,   then we apply the construction
of Case 1 (b) (i) in the leftward direction.

\item If  $\tilde{\pi}_{j-1}> \tilde{\pi}_{j+1}$,   then we apply the construction
of Case 1 (b) (ii) in the rightward direction.

\end{enumerate}
\end{enumerate}

\item Suppose that $\gamma_1={s(\sigma)+1}$, that is $\tilde{\pi}_j={s(\sigma)+1}$. 
\begin{enumerate}

\item If neither $\tilde{\pi}_{j-1}$ nor   $\tilde{\pi}_{j+1}$ belongs to  $(\overline{s(\sigma)+1}, {s(\sigma)+1})$,  then we define
$\widetilde{\tau}$ by replacing $s(\sigma)+1$  in $\widetilde{\pi}$ by $\overline{s(\sigma)+1}$ :
\[
\widetilde{\tau}
=
\widetilde{\pi}_1\cdots
\widetilde{\pi}_{j-1}\,
\overline{{s(\sigma)+1}}\,
\widetilde{\pi}_{j+1}\cdots
\widetilde{\pi}_n.
\]

 \item Suppose that  $\tilde{\pi}_{j-1} \in (\overline{s(\sigma)+1}, {s(\sigma)+1})$.  

\begin{enumerate}

\item If  $\tilde{\pi}_{j-1}< \tilde{\pi}_{j+1}$ and $\tilde{\pi}_{j+1} \in (\overline{s(\sigma)+1}, {s(\sigma)+1})$, then   choose maximal $r$ such that $1\leq r\leq n-j$ and $\tilde{\pi}_{j+i} \in (\overline{s(\sigma)+1}, \tilde{\pi}_{j+i-1})$ for $1\leq i\leq r$. Equivalently, $\tilde{\pi}_{j+r+1} \not \in (\overline{s(\sigma)+1}, \tilde{\pi}_{j+r})$. Define  
\[\tilde{\tau}=\tilde{\pi}_1 \,\cdots \tilde{\pi}_{j-1} \, \tilde{\pi}_{j+1} \, \cdots\, \tilde{\pi}_{j+r}\, \overline{s(\sigma)+1} \, \tilde{\pi}_{j+r+1}\,\cdots   \tilde{\pi}_{n}. \]

\item If  $\tilde{\pi}_{j-1}< \tilde{\pi}_{j+1}$ and $\tilde{\pi}_{j+1} \not \in (\overline{s(\sigma)+1}, {s(\sigma)+1})$,   then   choose maximal $r$ such that $1\leq r\leq j$ and $\tilde{\pi}_{j-i} \in (\overline{{s(\sigma)+1}}, \tilde{\pi}_{j-i+1} )$ for $1\leq i\leq r$. Equivalently,  $\tilde{\pi}_{j-r-1} \not \in (\overline{{s(\sigma)+1}}, \tilde{\pi}_{j-r})$. Define
\[\tilde{\tau}=\tilde{\pi}_1 \,\ldots \tilde{\pi}_{j-r-1} \, \overline{s(\sigma)+1} \, \tilde{\pi}_{j-r}\,\ldots \,\tilde{\pi}_{j-1} \,\tilde{\pi}_{j+1}\,\ldots\, \tilde{\pi}_{n}.\]

\item If  $\tilde{\pi}_{j-1}> \tilde{\pi}_{j+1}$, then we apply the construction of Case 2 (b) (ii) in the leftward direction.    

\end{enumerate}
\item Suppose that  $\tilde{\pi}_{j+1} \in (\overline{s(\sigma)+1}, {s(\sigma)+1})$. 

\begin{enumerate}

\item If  $\tilde{\pi}_{j-1}< \tilde{\pi}_{j+1}$,    then we apply the construction of Case 2 (b) (i) in the rightward direction.  

\item If  $\tilde{\pi}_{j-1}> \tilde{\pi}_{j+1}$ and $\tilde{\pi}_{j-1} \in (\overline{s(\sigma)+1}, {s(\sigma)+1})$,  then we apply the construction of Case 2 (b) (ii) in the leftward direction.    

\item If  $\tilde{\pi}_{j-1}> \tilde{\pi}_{j+1}$ and $\tilde{\pi}_{j-1} \not \in (\overline{s(\sigma)+1}, {s(\sigma)+1})$,   then we apply the construction of Case 2 (b) (i) in the rightward direction. 

\end{enumerate}
 
\end{enumerate}

\end{enumerate}
In each of the above cases, the word $\widetilde{\tau}$ is obtained from
$\widetilde{\pi}$ by changing exactly one letter, namely $\overline{s(\sigma)+1}$  to $s(\sigma)+1$  or
$s(\sigma)+1$  to $\overline{s(\sigma)+1}$ , and possibly moving this letter across a maximal string of
letters lying in the interval $(\overline{s(\sigma)+1},s(\sigma)+1)$. A direct comparison of adjacent
letters shows that $
\Des(\widetilde{\tau})=\Des(\widetilde{\pi}).
$ Moreover, $\widetilde{\tau}\in \Sh(\widetilde{\sigma},\Gamma(\sigma)).$ 
Hence, by Lemma~\ref{lem:inv}, $\Des(\tau)=\Des(\pi).$ Since exactly one sign has been changed, we also have $\nfix(\tau)\equiv \nfix(\pi)+1 \pmod 2.
$ Thus $\Psi$ is sign-reversing and descent-set-preserving.

Finally,   applying the same rule to $\tau$
changes the distinguished letter back and slides it to its original position.
Therefore, $\Psi(\Psi(\pi))=\pi.$ 
This proves that $\Psi$ is the desired involution and completes the proof. 
\end{proof}

For example, let $\pi=1\,6\,\bar3\,5\,8\,2\,7\,4\in E_8.$ Then
$\sigma=dp(\pi)=4\,6\,1\,5\,3\,2.$ 
We have
\[
\widetilde{\sigma}=\psi_8(\sigma)=6\,8\,1\,7\,3\,2
\]
and
\[
\widetilde{\pi}=\psi_8(\pi)
=
6\,8\,\bar4\,1\,7\,3\,5\,2
\in
\Sh(\widetilde{\sigma},\gamma),
\ \text{where}\  \gamma=\bar4\,5.
\]
It can be checked that $\widetilde{\pi}$ falls into Case~1(c) (ii), and the construction gives
\[
\widetilde{\tau}=6\,8\,1\,4\,7\,3\,5\,2.
\]
Thus
\[
\tau=\Psi(\pi)=5\,8\,1\,4\,6\,3\,7\,2.
\]
Hence $\tau\in E_8(\sigma)$, since $
dp(\tau)=dp(\pi)=4\,6\,1\,5\,3\,2.$ 
Moreover, $(-1)^{\nfix(\tau)}
=
-(-1)^{\nfix(\pi)}$, $
\Des(\pi)=\Des(\tau)=\{2,5,7\},$
and $\Psi(\tau)=\pi.$

The following table illustrates the involution
on $G_3=E_3\setminus \mathcal{D}_3$: 
\[
\begin{array}{c|l}
\toprule
\Des & \text{Pairings} \\
\midrule
\varnothing
& 123\leftrightarrow \bar{1}23 \\[2mm]
\{1\}
& 213\leftrightarrow 3\bar{2}1,
  \qquad 1\bar{2}3\leftrightarrow \bar{1}\bar{2}3 \\[2mm]
\{2\}
& 132\leftrightarrow \bar{1}32,
  \qquad 12\bar{3}\leftrightarrow \bar{1}2\bar{3} \\[2mm]
\{1,2\}
& 321\leftrightarrow 21\bar{3},
  \qquad 1\bar{2}\bar{3}\leftrightarrow \bar{1}\bar{2}\bar{3} \\
\bottomrule
\end{array}
\]

\section{Proofs of Proposition \ref{prop-deradesarra} and  Proposition \ref{prop:recurr}}

For a permutation \(\sigma=\sigma_1\cdots \sigma_n\), let
\[
a(\sigma)=\min\{i \in [n-1]\colon  \sigma^{-1}_i<\sigma^{-1}_{i+1}\},
\]
if such an ascent exists, and put \(a(\delta_n)=n\) for the decreasing word $\delta_n=n(n-1)\cdots 1.$ 
Set
\[
\mathcal F_n=\{\sigma \in \mathfrak{S}_n \colon \sigma^{-1} \in A^e_n\}.
\]
In other words, $\sigma \in F_n$ if and only if $a(\sigma)$ is even. 

It is easy to see that the proof of Proposition \ref{prop-deradesarra}  is equivalent to showing that

\begin{prop}\label{prop-deradesarra-v2}
For all $J \subseteq \{1, 2, \dots, n-1\}$,
\begin{equation}\label{prop-deradesarra-eq}
\bigl|\{\sigma \in \mathcal{D}_n \mid \Des(\sigma) = J\}\bigr|
= \bigl|\{\sigma \in \mathcal F_n \mid \Des(\sigma) = J\}\bigr|.
\end{equation}
\end{prop}
\noindent{\it Proof.} Let
    \[
    \mathfrak{S}_{n,k} = \bigcup_{\beta\in \mathfrak{S}(k+1,\dots,n)} \delta_k \sqcup \beta,
    \]
    and define
    \begin{equation*} 
    S_J = \sum_{k=0}^n (-1)^k \bigl| \bigl\{ \sigma \in \mathfrak{S}_{n,k} \,\big|\, \operatorname{Des}(\sigma)=J \bigr\} \bigr|.
    \end{equation*}
    We will show that $S_J$ coincides with both the left-hand side and the right-hand side of \eqref{prop-deradesarra-eq}, which yields the desired equality.
First, applying the descent-preserving bijection $\phi$ established in Section 2 and the involution $\Psi$ from the proof of Theorem \ref{main2a} in Section 3, we derive that 
    \[
    \begin{aligned}
    S_J &= \sum_{k=0}^n (-1)^k \bigl| \bigl\{ \sigma \in \mathfrak{S}_{n,k} \,\big|\, \operatorname{Des}(\sigma)=J \bigr\} \bigr| \\
    &= \sum_{k=0}^n (-1)^k \bigl| \bigl\{ \sigma \in E_{n,k} \,\big|\, \operatorname{Des}(\sigma)=J \bigr\} \bigr| \\
    &= \bigl| \bigl\{ \sigma\in \mathcal{D}_n \,\big|\, \operatorname{Des}(\sigma)=J \bigr\} \bigr|.
    \end{aligned}
    \]
On the other hand, it follows directly from definition that $\sigma\in \mathfrak{S}_{n,k}$ if and only if $\alpha(\sigma)\ge k$. Substituting this   gives
    \[
    \begin{aligned}
    S_J &= \sum_{k=0}^n (-1)^k \bigl| \bigl\{ \sigma \in \mathfrak{S}_{n,k} \,\big|\, \operatorname{Des}(\sigma)=J \bigr\} \bigr| \\
    &= \sum_{k=0}^n (-1)^k \bigl| \bigl\{ \sigma \in \mathfrak{S}_n \,\big|\, \alpha(\sigma)\ge k,\ \operatorname{Des}(\sigma)=J \bigr\} \bigr| \\
    &= \sum_{k\ge 0} \bigl| \bigl\{ \sigma\in \mathfrak{S}_n \,\big|\, \alpha(\sigma)\ge 2k,\ \operatorname{Des}(\sigma)=J \bigr\} \bigr| \\
    &\quad - \sum_{k\ge 0} \bigl| \bigl\{ \sigma\in \mathfrak{S}_n \,\big|\, \alpha(\sigma)\ge 2k+1,\ \operatorname{Des}(\sigma)=J \bigr\} \bigr| \\
    &= \sum_{k\ge 0} \bigl| \bigl\{ \sigma\in \mathfrak{S}_n \,\big|\, \alpha(\sigma)=2k,\ \operatorname{Des}(\sigma)=J \bigr\} \bigr| \\
    &= \bigl| \bigl\{ \sigma \in F_n \,\big|\, \operatorname{Des}(\sigma) = J \bigr\} \bigr|.
    \end{aligned}
    \]
This completes the proof. \qed

We conclude the paper by presenting combinatorial proofs for the recurrence relations in Proposition \ref{prop:recurr}, utilizing Proposition \ref{prop-deradesarra-v2} alongside the  insertion lemma for ordinary permutations. As a direct consequence of Proposition \ref{prop-deradesarra-v2}, the $q$-desarrangement number admits the equivalent combinatorial formulation
    \begin{equation}\label{eq:DW}
    e_n(q)=\sum_{w\in\mathcal F_n}q^{\maj(w)}.
    \end{equation}

We conclude this paper with combinatorial proofs of  Proposition \ref{prop:recurr} with the aid of Proposition \ref{prop-deradesarra-v2} and the insertion lemma for ordinary permutations. By Proposition \ref{prop-deradesarra-v2}, we see that   $q$-desarrangement number can be expressed as,
\begin{equation}\label{eq:DW}
e_n(q)=\sum_{w\in\mathcal F_n}q^{\maj(w)}.
\end{equation}

Let \(\sigma=\sigma_1\sigma_2\cdots \sigma_n\in \mathfrak{S}_n\).  There are \(n+1\) positions in \(\sigma\), namely the positions before \(\sigma_1\), between two consecutive letters, and after \(\sigma_m\).  We use the standard major-index labelling of these positions:

\begin{itemize}
  \item label the final position, after \(\sigma_n\), by \(0\);
  \item label the descent positions, namely the positions after positions \(i\) with \(\sigma_i>\sigma_{i+1}\), from right to left by
  \[
  1,2,\ldots,\des(\sigma);
  \]
  \item label the remaining positions from left to right by
  \[
  \des(\sigma)+1,\des(\sigma)+2,\ldots,n.
  \]
\end{itemize}

Let \(I_{n,r}(\sigma)\) denote the permutation obtained by inserting the new largest letter \(n+1\) into the position  labeled \(r\).  The classical   insertion lemma asserts that
\begin{equation}\label{eq:maj-insertion}
\maj(I_{n,r}(w))=\maj(w)+r,
\qquad 0\le r\le n,
\end{equation}
see, for example,  Gupta \cite{Gupta-1978},  Haglund, Loehr and Remmel \cite{Haglund-Loehr-Remmel-2005} and Han \cite{Han-1992}. 

We shall use the following elementary form of the desarrangement insertion property.

\begin{lem}\label{lem:desarrangement-insertion}
For \(n\ge 1\),  
\begin{equation}\label{eqn:desa-insert}
\{I_{n,r}(u):u\in\mathcal F_n,\,0\le r\le n\}
=
\begin{cases}
\mathcal F_{n+1}\setminus\{\delta_{n+1}\}, & \text{if } n \text{ is odd},\\[4pt]
\mathcal F_{n+1}\cup\{\delta_{n+1}\}, & \text{if }  n \text{ is even}.
\end{cases}
\end{equation}
\end{lem}

\begin{proof}
Let \(u\in \mathfrak{S}_n\), and insert the entry \(n+1\) into \(u\) to yield a permutation \(w\in \mathfrak{S}_{n+1}\). Passing to the inverse permutation \(w^{-1}\), this insertion operation equivalently adds a new position to \(u^{-1}\). Consequently, all pairwise comparisons
    \[
    u^{-1}_i< u^{-1}_{i+1}, \qquad 1\le i\le n-1,
    \]
    remain unchanged. It follows that the first ascent of \(u^{-1}\) is preserved under this insertion, with the sole exception of the strictly decreasing permutation. This decreasing permutation constitutes the only exceptional case. Recall that \(\delta_n \in \mathcal{F}_{n}\) if and only if \(n\) is even. As a result, the exceptional permutation \(\delta_{n+1}\) is excluded from the left-hand side of \eqref{eqn:desa-insert} when \(n+1\) is even, and appears as an additional permutation when \(n+1\) is odd. This completes the proof of the assertion.
\end{proof}

\noindent{\it Proof of  Proposition \ref{prop:recurr}. } It is enough to prove the recurrence \eqref{rec-eqa}  in terms of the combinatorial interpretation \eqref{eq:DW}, that is 
\begin{equation}\label{eq:recdesage}
e_n(q)=[n]_qe_{n-1}(q)+ (-1)^n q^{\binom{n}{2}}.
\end{equation}

We only treat the case where $n$ is odd; the even  case follows by analogous reasoning.
Observe that $\delta_n \notin \mathcal{F}_n$ while $\delta_{n-1} \in \mathcal{F}_{n-1}$.
Combining   \eqref{eq:maj-insertion} with Lemma \ref{lem:desarrangement-insertion}, we derive that 
\begin{align*}
e_{n}(q) &= \sum_{\sigma \in \mathcal{F}_{n} } q^{\operatorname{maj}(\sigma)} \\[5pt]
&= \sum_{u \in \mathcal{F}_{n-1} }\sum_{i=0}^{n-1}q^{\operatorname{maj}(I_{n-1,i}(u))} - q^{\operatorname{maj}(\delta_{n})} \\[5pt]
&= \bigl(1+q+\cdots+q^{n-1}\bigr) \sum_{u \in \mathcal{F}_{n-1} } q^{\operatorname{maj}(u)} - q^{\binom{n}{2}} \\[5pt]
&= [n]_q e_{n-1}(q)-q^{\binom{n}{2}}.
\end{align*}
This verifies the recurrence \eqref{rec-eqa}.

We now turn to the second recurrence \eqref{rec-eqb}. Our goal is to prove
\begin{equation}\label{rec-b}
e_n(q) = [n-1]_q\bigl(e_{n-1}(q) + q^{n-1} e_{n-2}(q)\bigr).
\end{equation}
We partition $\mathcal{F}_n$ into two disjoint subsets as follows.

The first set collects permutations obtained by inserting the value $n$ into any non-maximal major-index insertion position of elements in $\mathcal{F}_{n-1}$:
\[
\mathcal{C}_1 =
\bigl\{I_{n-1,r}(u) \,\big|\, u\in\mathcal{F}_{n-1},\ 0\le r\le n-2\bigr\}.
\]

The second set consists of permutations constructed in two insertion steps: first insert $n-1$ into an element of $\mathcal{F}_{n-2}$, then place $n$ at the maximal major-index insertion position:
\[
\mathcal{C}_2 =
\bigl\{I_{n-1,n-1}\bigl(I_{n-2,r}(v)\bigr) \,\big|\,
 v\in\mathcal{F}_{n-2},\ 0\le r\le n-2\bigr\}.
\]

We claim that
\begin{equation}\label{eq:decomposition}
\mathcal F_n=\mathcal C_1\sqcup\mathcal C_2.
\end{equation}
Indeed, take \(w\in\mathcal F_n\) and delete the letter \(n\).  Suppose that \(n\) had been inserted into the position labelled \(r\).  If \(0\le r\le n-2\), then by Lemma~\ref{lem:desarrangement-insertion} the remaining word belongs to \(\mathcal F_{n-1}\), and hence \(w\in\mathcal C_1\).  If \(r=n-1\), then deleting \(n\) leaves a word which, by applying Lemma~\ref{lem:desarrangement-insertion} one size lower, is uniquely of the form
\[
I_{n-2,r'}(v),
\qquad v\in\mathcal F_{n-2},\quad 0\le r'\le n-2.
\]
Thus \(w\in\mathcal C_2\).

The two classes are disjoint because the letter \(n\) is inserted in a non-maximal major-index position in \(\mathcal C_1\), whereas it is inserted in the maximal major-index position in \(\mathcal C_2\).  The exceptional decreasing words in Lemma~\ref{lem:desarrangement-insertion} are also accounted for correctly: if \(n\) is even, then \(\delta_n\in\mathcal F_n\) arises from \(\delta_{n-1}\), which appears as the exceptional image from \(\mathcal F_{n-2}\); if \(n\) is odd, then \(\delta_n\notin\mathcal F_n\), and the corresponding exceptional insertion is absent.  Hence \eqref{eq:decomposition} holds.

Using \eqref{eq:maj-insertion}, the contribution of \(\mathcal C_1\) is
\[
\sum_{u\in\mathcal F_{n-1}}\sum_{r=0}^{n-2}q^{\maj(u)+r}
=
[n-1]_q\sum_{u\in\mathcal F_{n-1}}q^{\maj(u)}
=
[n-1]_qe_{n-1}(q).
\]

Similarly, the contribution of \(\mathcal C_2\) is
\[
\begin{aligned}
\sum_{v\in\mathcal F_{n-2}}\sum_{r=0}^{n-2}
q^{\maj(I_{n-1,n-1}(I_{n-2,r}(v)))}
&=
\sum_{v\in\mathcal F_{n-2}}\sum_{r=0}^{n-2}
q^{\maj(v)+r+n-1} \\
&=q^{n-1}[n-1]_q
\sum_{v\in\mathcal F_{n-2}}q^{\maj(v)} \\
&=q^{n-1}[n-1]_qe_{n-2}(q).
\end{aligned}
\]
Summing the contributions of the disjoint sets $\mathcal{C}_1$ and $\mathcal{C}_2$, together with the identity \eqref{eq:DW}, yields   \eqref{rec-b}.
This completes the proof. \qed

 \vskip 0.2cm
\noindent{\bf Acknowledgment.} This work
was supported by   the National Science Foundation of China.  We are grateful to Prof. Catherine H.F. Yan for her insightful comments and suggestions.

\end{document}